\documentclass{article}
\usepackage{amsmath,amstext,amssymb,mathtools}
\usepackage{graphicx,graphics,color}
\usepackage{theorem,cite}
\usepackage{hyperref}
\usepackage{unicode_character}

\newtheorem{theorem}{Theorem}
\newtheorem{proposition}{Proposition}[section]
\newtheorem{corollary}[proposition]{Corollary}

\newtheorem{lemma}[proposition]{Lemma}
\theoremstyle{break} 

\newenvironment{proof}%
{{\par\noindent \bf Proof. \nobreak}}%
{\nobreak \removelastskip \nobreak \hfill $\Box$ \medbreak}
{{\par\noindent \bf Proof lemma. \nobreak}}%
{\nobreak \removelastskip \nobreak \bf End proof lemma. \medbreak}
\newenvironment{remark}{\par \medskip \noindent {\bf Remark. }\nobreak}{\par \medskip}

\numberwithin{equation}{section}

\newcommand{\expo}{\mathrm{e}}
\newcommand{\ds}{\displaystyle}

\newcommand{\dd}{\mathrm{d}}
\newcommand{\E}{\mathbb{E}}

\newcommand{\csch}{\text{csch}}

\newcommand{\cev}[1]{\reflectbox{\ensuremath{\vec{\reflectbox{\ensuremath{#1}}}}}}

\title{Generative diffusion models from a PDE perspective}

\author{Fei Cao \footnotemark[1] \and Kimball Johnston \footnotemark[2] \and Thomas Laurent \footnotemark[3] \and Justin Le \footnotemark[2] \and S\'ebastien Motsch \footnotemark[2]}

\begin{document}

\maketitle

\footnotetext[1]{University of Massachusetts Amherst - Department of Mathematics and Statistics, 710 N Pleasant St, Amherst, MA 01003, USA}
\footnotetext[2]{Arizona State University - School of Mathematical and Statistical Sciences, 900 S Palm Walk, Tempe, AZ 85287, USA}
\footnotetext[3]{Loyola Marymount University, USA}

\tableofcontents

\begin{abstract}
Diffusion models have become the de facto framework for generating new datasets. The core of these models lies in the ability to reverse a diffusion process in time. The goal of this manuscript is to explain, from a PDE perspective, how this method works and how to derive the PDE governing the reverse dynamics as well as to study its solution analytically. By linking forward and reverse dynamics, we show that the reverse process's distribution has its support contained within the original distribution. Consequently, diffusion methods, in their analytical formulation, do not inherently regularize the original distribution, and thus, there is no generalization principle. This raises a question: where does generalization arise, given that in practice it does occur?
Moreover, we derive an explicit solution to the reverse process's SDE under the assumption that the starting point of the forward process is fixed. This provides a new derivation that links two popular approaches to generative diffusion models: stable diffusion (discrete dynamics) and the score-based approach (continuous dynamics). Finally, we explore the case where the original distribution consists of a finite set of data points. In this scenario, the reverse dynamics are explicit (i.e., the loss function has a clear minimizer), and solving the dynamics fails to generate new samples: the dynamics converge to the original samples. In a sense, solving the minimization problem exactly is {\it ``too good for its own good"} (i.e., an overfitting regime).
\end{abstract}

\section{Introduction}
\setcounter{equation}{0}

In recent years, generative diffusion models such as score-based generative modeling (SGM) and denoising diffusion probabilistic models (DDPM) have witnessed tremendous success in applications including but not limited to image generation \cite{dhariwal2021diffusion, saharia2022palette}, text generation \cite{li2022diffusion, gong2022diffuseq}, audio generation \cite{pmlr-v139-popov21a, kong2021diffwave}, video generation \cite{ho2022video}, and more \cite{poole2023dreamfusion, baranchuk2022labelefficient}. The ever-growing list of relevant literature is vast, and we refer to \cite{nichol2021improved,chen2023improved,chen2023sampling,ho2020denoising,sohl2015deep,song2020improved,song2019generative,song2020score,weng2021diffusion,dieleman2023perspectives,alexandre2023denoising} and references therein for a general overview of the diffusion models in machine learning.

In the most elementary scenario, the goal of a generative model is to produce a collection of new samples, denoted by $\{\widetilde{\bf x}_i\}_{i=1}^M$, based solely on a set of available data, denoted by $\{{\bf x}_i\}_{i=1}^N$. The aim is for the generated samples to closely resemble the original data while still being distinct\footnote{The Thai expression {\it ``same same, but different"} captures this concept well}. Typically, we assume that our available samples $\{{\bf x}_i\}_{i=1}^N$ are i.i.d. (independent and identically distributed) and follow an unknown probability distribution $\rho_0 \in \mathcal{P}(\mathbb{R}^d)$ (see figure \ref{fig:intro_generated_q}). A popular approach in statistics is to use a kernel density estimation technique to reconstruct an approximation, denoted by $q$, of $ρ_0$ from the original sample \cite{silverman2018density} (see figure \ref{fig:intro_generated_q}-A). Afterwards, samples are generated from $q$. This approach is quite effective in low dimensions but becomes ineffective in high dimensions (e.g., images) since one has to learn the structure of the data first.

To face the curse of dimensionality, a different approach has emerged using (variational) auto-encoders. The idea is to first {\it encode} or project each sample onto a latent space of lower dimension. This latent space can be interpreted as the key features of the sample. Afterwards, the features are {\it decoded} in a second step to generate a sample (see figure \ref{fig:intro_generated_q}-B).  Notice that in this approach, new samples are generated without the need to construct an approximate density $q$. Diffusion models operate similarly, as they also involve an encoding-decoding process. However, in this case, the encoding phase consists of gradually adding noise to each sample, transforming the (unknown) distribution $\rho_0$ into a standard Gaussian density $\mathcal{N}$. Then, a denoising procedure is performed to reconstruct the original data $\rho_0$ from pure noise by reversing (in time) the forward diffusion process. It is worth noting that the encoding in diffusion models is explicit (in the sense that it requires no training) and does not involve dimensionality reduction. The advantage of this approach is that there is no need to fix a-priori the dimensionality of the latent space. However, the method may be prone to overfitting (e.g., the generated dataset could end up being a mere replication of the training data).

\begin{figure}[ht]
  \centering
  \includegraphics[width=1.0\textwidth]{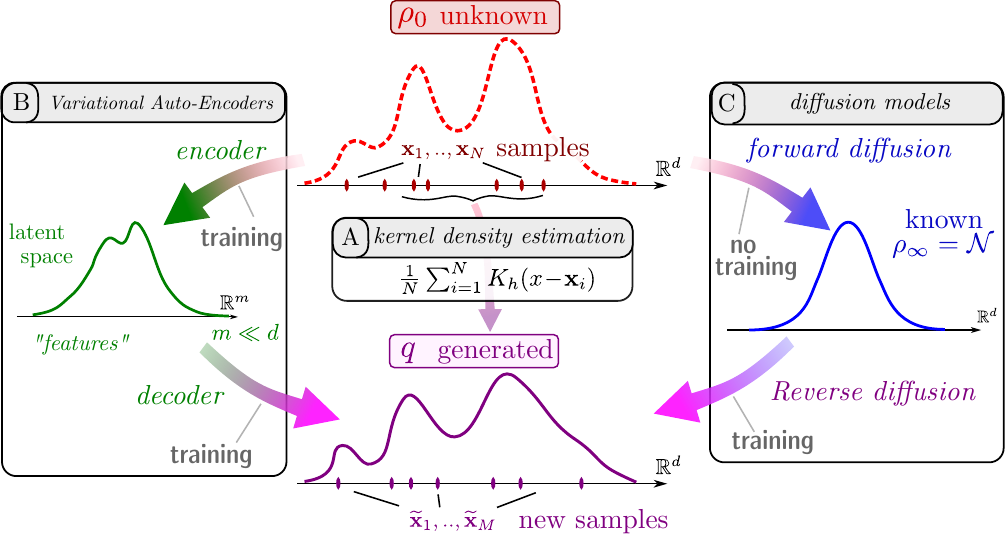}
  \caption{Illustration of various generative models. Given a data sample $\{{\bf x}_i\}_{i=1}^N$ from an unknown distribution $ρ_0$, the goal is to generate new samples $\{\widetilde{\bf x}_i\}_{i=1}^M$. In kernel density estimation (A), a density $q$ is first estimated, and new samples are then drawn from this estimated distribution. In the auto-encoder approach (B), an encoder-decoder network is trained to reconstruct the samples. Diffusion models (C) work similarly to the auto-encoder method, except only the decoder must be learned, and the encoder does not reduce dimensionality.}
  \label{fig:intro_generated_q}
\end{figure}

%
%

Despite the existence of a vast literature on diffusion models and related generative models, a majority of the models  require a heavy or advanced background in probability theory and SDEs to explain, design, adjust, and optimize (though some works derive these models from other points of view, including mean-field-games  \cite{zhang2023mean} and optimal transport \cite{de2021diffusion}). One of our main goals in the present manuscript is the re-introduction of the basic diffusion model from a PDE perspective. We believe that this pedestrian's guide to diffusion models, which highlights their PDE ingredients (and is also suitable for pedagogical purposes), will inspire novel research efforts to investigate the subtle power of such generative models from a PDE viewpoint.


The rest of the manuscript is organized as follows: in section \ref{sec:sec2}, we provide a PDE-based approach for the derivation of the reverse diffusion process from the forward diffusion process (i.e., the noising process that turns the data into Gaussian noise). Section \ref{sec:sec3} is devoted to a more traditional derivation of the reversed diffusion process from an SDE perspective. Our derivations in sections \ref{sec:sec2} and \ref{sec:sec3} are sufficiently elementary compared to the technical papers \cite{anderson1982reverse,haussmann1986time}. In section \ref{sec:sec4}, we convey the message that a perfect (reversed) diffusion process will not generalize well. This fact seems to be well-recognized in the machine learning community \cite{li2024good, Somepalli_2023_CVPR}, but we failed to locate a reference offering the precise, mathematical explanation of this phenomenon presented here. The argument presented in section \ref{sec:sec4} is entirely analytic (as opposed to probabilistic), which reveals the appealing advantage of a PDE approach. We provide a few numerical implementations of the diffusion models in section \ref{sec:sec5}, where the link between score-based generative modeling and stable diffusion is emphasized as well. Lastly, we give our conclusion in section \ref{sec:conclusion} along with a list of important open questions to be addressed in subsequent works.



\section{Reverse diffusion from PDE and SDE perspectives}

\subsection{A PDE approach to reverse diffusion}
\label{sec:sec2}
\setcounter{equation}{0}

\subsubsection{Derivation of the PDE}

We would like to introduce and derive the reverse PDE dynamics used in diffusion models. We will adopt a PDE approach as opposed to the SDE approach frequently encountered in relevant literature \cite{anderson1982reverse,haussmann1986time}. An overview of the problem is given in figure \ref{fig:X_rho_PDE}. There are two processes at play. First, the forward diffusion process transforms a density distribution $ρ_0$ into a standard normal distribution $\mathcal{N}$:
\begin{equation}
  \hspace{-1.5cm} \begin{array}{c}
    \text{\bf Forward}\\
    \text{\bf diffusion}
    \end{array} \quad \qquad \left\{ \begin{array}{l}
      ∂_t ρ = ∇⋅(xρ) + Δρ \\
      ρ(t=0) = ρ_0.
    \end{array}
  \right.
  \label{eq:forward_PDE}
\end{equation}
This process is known as the Ornstein-Uhlenbeck equation (or the Langevin equation). The law $ρ$ converges exponentially quickly toward the normal distribution under various metrics \cite{bakry2014analysis}. In practice, the forward diffusion is only solved up to a finite time $T_*$, so the resulting distribution $ρ(.,T_*)$ is only exponentially close to the equilibrium distribution $\mathcal{N}$.

\begin{figure}[ht]
  \centering
  \includegraphics[width=.7\textwidth]{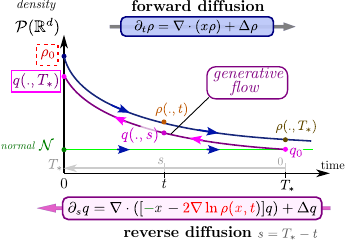}
  \caption{The forward process as a PDE \eqref{eq:forward_PDE} transforms any distribution $ρ_0$ into a normal distribution $\mathcal{N}$. The reverse process does the opposite and transforms a normal distribution $\mathcal{N}$ into (almost) $ρ_0$.}
  \label{fig:X_rho_PDE}
\end{figure}

The second step is to {\it reverse} the dynamics to transform a normal distribution $\mathcal{N}$ back into $ρ_0$. A naive way to carry out this step is simply to look at the evolution of $ρ(.,t)$ in reverse time by introducing the distribution $q(.,s)$ with
\begin{equation}
  \label{eq:reverse_rho}
  q(x,s) \coloneqq ρ(x,T_*-s) \quad \text{for } 0 \leq s \leq T_*.
\end{equation}
The starting point $s=0$ corresponds to the final time $t=T_*$. This distribution $q$ satisfies
\begin{equation}
  \label{eq:naive_pde}
  \hspace{-1.5cm} \begin{array}{c}
    \text{\bf Unstable}\\
    \text{\bf Reverse diffusion}
    \end{array} \quad \quad   \left\{
    \begin{array}{l}
      ∂_s q = -∇⋅(xq) - Δq \\
      q(x,s=0) = ρ(x,T_*).
    \end{array}
  \right.
\end{equation}
This PDE is ill-posed due to the anti-diffusion operator $-Δ$ (see figure \ref{fig:stable_reverse}). To make the PDE stable while preserving the property \eqref{eq:reverse_rho}, the key ingredient is to rewrite the diffusion process $Δ$ as a (non-linear) transport equation \cite{degond1989weighted}:
\begin{equation}
  \label{eq:diffusion_transport}
  Δφ = ∇⋅(∇φ) = ∇⋅(φ∇ \log φ).
\end{equation}
We then deduce
\begin{equation}
  \label{eq:trick_thomas}
  Δφ = 2∇⋅(φ∇ \log φ) - Δφ.
\end{equation}
If $q$ satisfies \eqref{eq:reverse_rho}, we obtain
\begin{equation}
  \label{eq:minus_delta_q}
  -Δq(x,s) =  -2∇⋅(q(x,s)∇ \log ρ(x,T_*-s)) + Δq(x,s).
\end{equation}

\begin{figure}[ht]
  \centering
  \includegraphics[width=.99\textwidth]{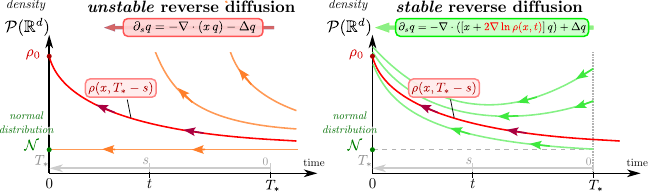}
  \caption{Reversing the forward process \eqref{eq:forward_PDE}  by reversing time in a naive way \eqref{eq:reverse_rho} yields an ill-posed problem \eqref{eq:naive_pde} even though the curve $s→ρ(x,T_*-s)$ is a solution.
  The {\it true} reverse dynamics \eqref{eq:reverse_diffusion} preserves this curve but also stabilizes the solution. As a consequence, the normal distribution $\mathcal{N}$ is no longer an equilibrium distribution for the (non-autonomous) reverse dynamics.}
  \label{fig:stable_reverse}
\end{figure}

This motivates the introduction of the reverse diffusion equation for $0\leq s \leq T_*$:
\begin{equation}
  \label{eq:reverse_diffusion}
   \begin{array}{c}
    \text{\bf Stable}\\
    \text{\bf Reverse diffusion}
    \end{array} \quad
  \left\{
    \begin{array}{l}
      ∂_s q = ∇⋅\left(\left[-x - 2∇ \log ρ(.,T_*-s)\right]q\right) + Δq \\
      q(x,s=0) = q_0(x).
    \end{array}
  \right.
\end{equation}
In short, we have stabilized the unstable reverse diffusion \eqref{eq:naive_pde} by adding twice the Laplace operator $Δ$ and removing twice the non-linear transport term. In particular, if we take $q_0(x) = ρ(x,T_*)$, then $q(x,s) = ρ(x,T_*-s)$ is the solution to the reverse diffusion \eqref{eq:reverse_diffusion}.

\subsubsection{Generalization}

There are other ways to stabilize the unstable reverse diffusion \eqref{eq:naive_pde}. For instance, one can add the Laplace operator $Δ$ and remove the non-linear transport term just once each, leading to the following PDE:
\begin{equation}
  ∂_s q = ∇⋅\big([-x - ∇ \log ρ(.,T_*-s)]q\big).
  \label{eq:reverse_no_diff}
\end{equation}
As there is no diffusion operator, the associated stochastic process will be deterministic (see further discussions in \cite{karras2022elucidating}).

Another generalization is to consider a different forward operator. Instead of using the Ornstein-Uhlenbeck equation \eqref{eq:forward_PDE}, one can simply use a Brownian motion and its associated generator $Δ$, leading to the following forward and reverse PDEs:
\begin{eqnarray}
  \label{eq:forward_basic_diffusion}
  &&∂_t ρ = Δρ,\\
  &&∂_s q = -∇⋅\big(2∇ \log ρ(.,T_*-s)\,q\big) + Δq.
\end{eqnarray}
Although this formulation simplifies the computations compared with the Ornstein-Uhlenbeck dynamics, there is nonetheless an additional requirement: one needs to introduce a bounded domain with boundary conditions so that the forward dynamics converges to a uniform equilibrium. Without the presence of a boundary, the solution to \eqref{eq:forward_basic_diffusion} will converge to the zero function \cite{evans2022partial}.

\subsection{An SDE approach to reverse diffusion}\label{sec:sec3}

We now present a more traditional derivation of the reverse dynamics from the SDE perspective. An overview of this approach is encapsulated in figure \ref{fig:X_rho_SDE}.

\begin{figure}[ht]
  \centering
  \includegraphics[width=.7\textwidth]{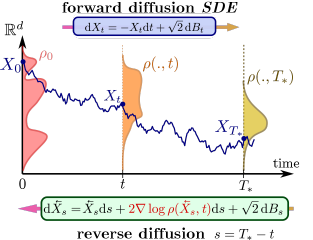}
  \caption{The forward process as an SDE \eqref{eq:SDE} transforms the initial law $\rho_0$ into a standard normal distribution $\mathcal{N}$, while the reverse process performs the opposite transformation.}
  \label{fig:X_rho_SDE}
\end{figure}

\subsubsection{Forward SDE}

In this subsection, we review the basic properties of the forward diffusion process. Consider a stochastic process $X_t$ in $ℝ^d$ satisfying the Ornstein-Uhlenbeck dynamics:
\begin{equation}
  \label{eq:SDE}
  \dd X_t = -X_t\,\dd t + \sqrt{2}\,\dd B_t,
\end{equation}
where $B_t$ is the standard Brownian motion in $ℝ^d$. Using an integrating factor, we obtain the explicit solution
\begin{equation}
  \label{eq:sol_OU}
  X_t = \mathrm{e}^{-t}\,X_0 + \sqrt{2} \mathrm{e}^{-t} ∫_{0}^t \mathrm{e}^z \,\dd B_z  \;\;\sim \mathcal{N}(\mathrm{e}^{-t}X_0,σ^2_t \text{Id}),
\end{equation}
where $\mathcal{N}$ is a Gaussian distribution and $σ^2_t \coloneqq 1-\mathrm{e}^{-2t}$. Thanks to this explicit formula, we can show the convergence of $X_t$ toward the standard normal distribution $\mathcal{N}({\bf 0},\text{Id})$ in various senses \cite{bakry2014analysis}.

The process $X_t$ naturally induces an operator $L$ (its {\it generator}) \cite{evans2012introduction,oksendal2013stochastic}
\begin{equation}
  \label{eq:L}
  L[φ](x) = -x⋅∇φ(x) + Δφ(x).
\end{equation}
In particular, for a given test function $φ$, if we consider the function 
\begin{equation}
  \label{eq:u}
  u(x,t) = \E[φ(X_t) \mid X_0=x],
\end{equation}
then it follows from It\^{o}'s formula that $u$ satisfies the Kolmogorov {\it backward} equation:
\begin{equation}
  \label{eq:kolmo_backward}
  \left\{
    \begin{array}{l}
      ∂_t u = L[u] \\
      u(x,t=0) = φ(x).
    \end{array}
  \right.
\end{equation}

\begin{remark}
  The terminology {\it backward} here might be confusing. The Kolmogorov {\it backward} equation \eqref{eq:kolmo_backward} follows the process $X_t$, while the law $ρ(x,t)$ of $X_t$ satisfies the Kolmogorov {\it forward} equation:
  \begin{equation}
    \label{eq:L_star}
    ∂_t ρ = L^*[ρ], \qquad L^*[φ](x) \coloneqq ∇⋅\big(x\,φ(x)\big) + Δφ(x),
  \end{equation}
  which is precisely \eqref{eq:forward_PDE}. Both equations evolve {\it forward} in time.
\end{remark}

To conclude our review of the forward process, we consider the process with a terminal condition at a fixed time $T_*$:
\begin{equation}
  \label{eq:v}
  v(x,t) = \E[φ(X_{T_*}) \mid X_t=x].
\end{equation}
This satisfies the Feynman-Kac formula \cite{evans2012introduction}:
\begin{equation}
  \label{eq:Feynmann_Kac}
  \left\{
    \begin{array}{l}
      ∂_t v = -L[v] \\
      v(x,t=T_*) = φ(x).
    \end{array}
  \right.
\end{equation}
We will need this property to derive the SDE for the reserve process.

\subsubsection{Reverse SDE}

Now we are ready to derive the reverse diffusion equation \eqref{eq:reverse_diffusion} using stochastic processes. We consider the reverse (in time) trajectories
\begin{equation}
  \label{eq:reverse_SDE}
  \cev{X}_s = X_{T_*-s},
\end{equation}
where $X_t$ is the solution to the forward diffusion \eqref{eq:SDE}. Our goal is to find the generator associated with the processes $\cev{X}_s$. The traditional approach to finding the generator is to consider
\begin{equation}
  \label{eq:u_reverse}
  \cev{u}(x,s) = \E[φ(\cev{X}_s) \mid \cev{X}_0=x]
\end{equation}
and then take the time derivative. However, in this case, it is more convenient to consider the function with a terminal condition:
\begin{eqnarray}
  \label{eq:v_reverse}
  \cev{v}(x,s) &=& \E[φ(\cev{X}_{T_*}) \,\mid\, \cev{X}_s=x] \\
  \label{eq:v_reverse2}
               &=& \E[φ(X_{0}) \,\mid\, X_t=x],
\end{eqnarray}
using \eqref{eq:reverse_SDE} with $t \coloneqq T_*-s$. Notice that we first have to fix the initial condition $ρ_0$ for the law of $X_0$. To write the function $\cev{v}$ explicitly, we introduce the joint law of the pairwise process:
\begin{equation}
  \label{eq:joint_law}
  ρ(x,t_1\,;\,y,t_2) = \text{``joint law of $X_{t_1}$ and $X_{t_2}$''}.
\end{equation}
In particular,
\begin{equation}
  \label{eq:joint_law_at_0}
  ρ(x,0\,;\,y,t) = ρ(y,t \mid x,0)⋅ρ_0(x).
\end{equation}
We can now derive the equation satisfied by $\cev{v}$.

\begin{proposition}
  \label{ppo:cev_v}
  The function $\cev{v}(x,s)$ \eqref{eq:v_reverse} satisfies the following Kolmogorov backward equation: 
  \begin{equation}
    \label{eq:cev_v_pde}
    \left\{
      \begin{array}{l}
        ∂_s \cev{v} = \left(-x - 2∇ \log ρ(x,T_*-s)\right)⋅∇\cev{v} - Δ\cev{v} \\
        \cev{v}(x,s=T_*) = φ(x).
      \end{array}
    \right.
  \end{equation}
\end{proposition}
\begin{proof}
  Let's start by writing $\cev{v}$ using the joint law \eqref{eq:joint_law}:
  \begin{eqnarray}
    \cev{v}(x,s) &=& ∫_{y∈ℝ^d} φ(y)ρ(y,0  \mid x,t)\,\dd y \\
                 &=& ∫_{y∈ℝ^d} φ(y)ρ(x,t  \mid y,0) \frac{ρ_0(y)}{ρ(x,t)}\,\dd y
  \end{eqnarray}
  using Bayes' formula. This leads us to
  \begin{equation}
    \label{eq:cev_v_equal}
    ρ(x,t)\,\cev{v}(x,s) = ∫_{y∈ℝ^d} φ(y)\,ρ(x,t  \mid y,0)\,ρ_0(y) \,\dd y.
  \end{equation}
  Taking the derivative with respect to $s$ yields
  \begin{eqnarray}
    -∂_tρ\, \cev{v}\;+\; ρ\,∂_s \cev{v} = -∫_{y∈ℝ^d} φ(y)\,∂_tρ(x,t  \mid y,0)\,ρ_0(y) \,\dd y.
  \end{eqnarray}
  Using the fact that $ρ$ satisfies the Kolmogorov forward equation \eqref{eq:L_star}, we obtain
  \begin{eqnarray}
    -L^*[ρ]\, \cev{v}\;+\; ρ\,∂_s \cev{v} = -∫_{y∈ℝ^d} φ(y)\,L^*[ρ(x,t  \mid y,0)]\, ρ_0(y) \,\dd y.
  \end{eqnarray}
  Using the identity \eqref{eq:cev_v_equal}, we arrive at
  \begin{eqnarray}
    -L^*[ρ]\, \cev{v}\;+\; ρ\,∂_s \cev{v} = -L^*[ρ \cev{v}].
  \end{eqnarray}
  We then conclude thanks to Lemma \ref{lem:L_product} below.
\end{proof}

\begin{lemma}
  \label{lem:L_product}
  For any smooth functions $p$ and $φ$,
  \begin{equation}
    \label{eq:L_product}
    L^*[pφ] = L^*[p]φ + (xp+2∇p)\cdot ∇φ + pΔφ.
  \end{equation}
\end{lemma}
\begin{proof}
  The proof follows from direct computation:
  \begin{eqnarray}
    L^*[pφ] &=& ∇⋅(xpφ) + Δ(pφ) \\
            &=& ∇⋅(xp)φ + xp⋅∇φ + Δp\,φ + 2∇p⋅∇φ + pΔφ.
  \end{eqnarray}
\end{proof}
From Proposition \ref{ppo:cev_v} and the Feynman-Kac formula \eqref{eq:Feynmann_Kac}, we conclude that the generator of the process $\cev{X}_s$ is a time-dependent operator:
\begin{equation}
  \label{eq:R}
  R[φ(x),s] \coloneqq \left(x + 2∇ \log ρ(x,T_*-s)\right)⋅∇φ(x) + Δφ(x).
\end{equation}
We then deduce that the law $q(x,s)$ of the process $\cev{X}_s$ satisfies
\begin{equation}
  \label{eq:ds_q}
  ∂_s q(x,s) = R^*[q(x,s),s],
\end{equation}
with the operator
\begin{equation}
  \label{eq:R_adjoint}
  R^*[q(x),s] \coloneqq ∇⋅\Big( \big(-x - 2∇ \log ρ(x,T_*-s)\big)\,q(x)\Big) + Δq(x),
\end{equation}
thus recovering the reverse diffusion equation \eqref{eq:reverse_diffusion}.

\section{Reverse diffusion does {\it not} generalize}\label{sec:sec4}
\setcounter{equation}{0}


\subsection{The past of \texorpdfstring{$ρ$}{rho} is the future of \texorpdfstring{$q$}{q}}


Let's consider the forward and reverse dynamics together. Fix two initial distributions $ρ_0$ and $q_0$ along with a horizon time $T_*$, and then consider the forward process $X_t$ with generator $L$ \eqref{eq:L} and the reverse process $\cev{X}_s$ for $0\leq s \leq T_*$ with generator $R$ \eqref{eq:R}. Their respective laws $ρ(x,t)$ and $q(x,s)$ are solutions to
\begin{equation}
  \label{eq:forward_reverse}
  \left\{
    \begin{array}{l}
      ∂_t ρ = L^*[ρ] \\
      ρ(x,t=0) = ρ_0(x)
    \end{array}
  \right. \qquad \textrm{and} \qquad  \left\{
    \begin{array}{l}
      ∂_s q = R^*[q,s] \\
      q(x,s=0) = q_0(x),
    \end{array}
  \right.
\end{equation}
where the operators $L^*$ and $R^*$ are defined by \eqref{eq:L_star} and \eqref{eq:R_adjoint}, respectively.

As we have already introduced the joint law $ρ(. ,t_1;. ,t_2)$ of the forward process \eqref{eq:joint_law}, we now denote the joint law for the reverse process as
\begin{equation}
  \label{eq:joint_law_q}
  q(x,s_1\,;\,y,s_2) = \text{``joint law of $\cev{X}_{s_1}$ and $\cev{X}_{s_2}$''}.
\end{equation}

\begin{remark}
  Notice that the pairwise distribution $ρ(.;.)$ is time-homogeneous while $q(.;.)$ is not. In other words, we have the property that for $t_1<t_2$ and any $t_0$,
  \begin{equation}
    ρ(y,t_2 \mid x,t_1) = ρ(y,t_2+t_0 \mid x,t_1+t_0).
  \end{equation}
  This is not the case for $q(.\mid.)$, which reflects the time dependency of its generator $R$ \eqref{eq:R}.
\end{remark}

We can link the two joint laws \eqref{eq:joint_law} and \eqref{eq:joint_law_q}: backtracking into the past of $ρ$ is equivalent to looking into the future of $q$ (see the illustration in figure \ref{fig:past_future}).

\begin{figure}[ht]
  \centering
  \includegraphics[width=.99\textwidth]{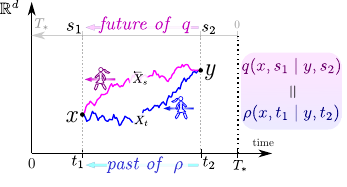}
  \caption{Illustration of Theorem \ref{thm:past_future} and the identity \eqref{eq:past_future}. The (density of the) probability that $X_{t_1}=x$ (past) knowing that $X_{t_2}=y$ (future) is the same as the probability that $\cev{X}_{s_1}=x$ (future) knowing that $\cev{X}_{s_2}=y$ (past).}
  \label{fig:past_future}
\end{figure}

\begin{theorem}
  \label{thm:past_future}
  For any $0\leq t_1 < t_2 \leq T_*$ and any $x,y∈ℝ^d$, we have
  \begin{equation}
    \label{eq:past_future}
    ρ(x,t_1 \mid y,t_2) = q(x,s_1 \mid y,s_2),
  \end{equation}
  with $s_1=T_*-t_1$ and $s_2=T_*-t_2$.
\end{theorem}
\begin{proof}
  Fix $y∈ℝ^d$ and $t_1∈[0,T_*]$, and take a test function $φ$. Consider $s$ and $t$ such that $t_1\leq t = T_* - s \leq T_*$, meaning $0 \leq s \leq T_* - t_1 = s_1$. Define the function
  \begin{eqnarray}
    a(y,s) &=& ∫_{x∈ℝ^d} φ(x) ρ(x,t_1 \mid y,T_* - s)\,\dd x \\
           &=& \mathbb{E}[φ(X_{t_1})  \mid  X_{T_* - s}=y] \\
           &=& \mathbb{E}[φ(X_{t_1})  \mid  X_t=y].
  \end{eqnarray}
  Following the same strategy as in the proof of Proposition \ref{ppo:cev_v}, we find that
  \begin{equation}
    \partial_s a(y, s) = -R[a(y, s), s].
  \end{equation}
  Thus, we have that $a$ satisfies
  \begin{equation}
    \label{eq:pde_a}
    \left\{
      \begin{array}{l}
        ∂_s a(y,s) = -R[a(y,s),s] \\
        a(y,s = T_* - t_1 = s_1) = φ(y),
      \end{array}
    \right.
  \end{equation}
  where $0 \leq s \leq s_1$.

  Now, an application of the Feynman-Kac formula (see for instance Proposition 3.41 in \cite{pardoux2014stochastic}) yields
  \begin{equation}
    a(y,s) = \mathbb{E}[φ(Y_{s_1})  \mid  Y_s=y],
  \end{equation}
  where $(Y_s)$ is a Markov process with generator $R$. On the other hand, we have already seen that $(\cev{X}_s)$ is exactly such a process. Hence, we conclude
  \begin{equation}
    a(y,s) = \mathbb{E}[φ(\cev{X}_{s_1})  \mid  \cev{X}_s=y].
  \end{equation}
  Finally, using the definition of $a$, we deduce that
  \begin{equation}\label{eq:conclusion}
    \mathbb{E}[φ(X_{t_1})  \mid  X_{t}=y] = \mathbb{E}[φ(\cev{X}_{s_1})  \mid  \cev{X}_s=y],
  \end{equation}
  where $t_1 = T_* - s_1$ and $t = T_* - s$.

  If we set $t_2 = t$ and $s_2 = s = T_* - t_2$, then since the equation \eqref{eq:conclusion} is satisfied for any test function $φ$, we arrive at the advertised identity \eqref{eq:past_future}.
\end{proof}

\subsection{Generated distribution \texorpdfstring{$q$}{q}}\label{sec:gen_distribution}

Thanks to Theorem \ref{thm:past_future}, we deduce the behavior of the reverse process $\cev{X}_s$ as $s$ approaches $T_*$ (i.e., when $t→0$).

\begin{corollary}
  For any $x,y∈ℝ^d$, the pairwise distributions of $ρ$ \eqref{eq:joint_law} and $q$ \eqref{eq:joint_law_q} satisfy
  \begin{equation}
    \label{eq:q_T_star_x0}
    q(y,T_* \mid x,0) =  ρ_0(y)⋅\frac{ρ(x,T_* \mid y,0)}{ρ(x,T_*)}.
  \end{equation}
\end{corollary}
\begin{proof}
  Let's again consider the function $\cev{u}$ \eqref{eq:u_reverse} with a test function $φ$:
  \begin{eqnarray}
    \cev{u}(x,s) &=& \E[φ(\cev{X}_s)  \mid  \cev{X}_0=x] = ∫_{y∈ℝ^d} φ(y)\, q(y,s\mid x,0)\,\dd y \\
                 &=& ∫_{y∈ℝ^d} φ(y)\, ρ(y,t\mid x,T_*)\,\dd y \\
                 &=& ∫_{y∈ℝ^d} φ(y)\, ρ(x,T_*\mid y,t)\frac{ρ(y,t)}{ρ(x,T_*)}\,\dd y \\
                 &\stackrel{s→T_*}{\longrightarrow}& ∫_{y∈ℝ^d} φ(y) ρ(x,T_*\mid y,0)\frac{ρ_0(y)}{ρ(x,T_*)}\,\dd y.
  \end{eqnarray}
  This completes the proof of \eqref{eq:q_T_star_x0}.
\end{proof}
As the law $ρ(x,T_*\mid y,0)$ is entirely explicit, i.e., $ρ(.,T_*\mid y,0) \sim \mathcal{N}(\expo^{-T_*}y,σ^2_{T_*}\text{Id})$ with $σ_{T_*}^2 = 1 -\expo^{-2T_*}$, we can even write an explicit formula:
\begin{equation}
  \label{eq:explicit_q_y_T_x_0}
  q(y,T_*\mid x,0) =  ρ_0(y) \frac{\exp\left(-\frac{|y\expo^{-T_*}-x|^2}{2σ_{T_*}^2}\right)}{∫_{z∈ℝ^d}ρ_0(z)\exp\left(-\frac{|z\expo^{-T_*}-x|^2}{2σ_{T_*}^2}\right)\,\dd z}.
\end{equation}

\begin{corollary}
  \label{cor:law_q_T_star}
  The law of $\cev{X}_s$ at the final time $s=T_*$ is given by
  \begin{equation}
    \label{eq:final_q}
    q(y,T_*) = ρ_0(y) ∫_{x∈ℝ^d} ρ(x,T_* \mid y,0) \frac{q_0(x)}{ρ(x,T_*)}\,\dd x.
  \end{equation}
\end{corollary}
\begin{proof}
  We simply notice that
  \begin{equation}
    q(y,T_*)  = ∫_{x∈ℝ^d} q(y,T_* \mid x,0)\, q_0(x)\,\dd x
  \end{equation}
  and then combine it with \eqref{eq:q_T_star_x0}.
\end{proof}
We notice that if we choose $q_0(x)=ρ(x,T_*)$, then we recover the initial distribution for the forward process, in the sense that $q(y,T_*) = ρ_0(y)$.

\begin{remark}
As an immediate by-product of Corollary \ref{cor:law_q_T_star}, we deduce the pointwise convergence of $q(y,T_*)$ to $\rho_0(y)$ as $T_*→+∞$ (at least for smooth $\rho_0$). Other works have investigated bounds on the distance between $q(.,T_*)$ and $\rho_0$ under various hypothesis in a variety of senses, such as KL divergence \cite{song2021maximum}, total variation \cite{de2021diffusion}, and types of integral probability metrics \cite{mimikos2024score}. 
\end{remark}

On the other hand, the main consequence of Corollary \ref{cor:law_q_T_star} is the realization that the {\it exact} reverse diffusion does {\it not} generalize (see the illustration in figure \ref{fig:support_q_T}).

\begin{figure}[ht]
  \centering
  \includegraphics[width=.99\textwidth]{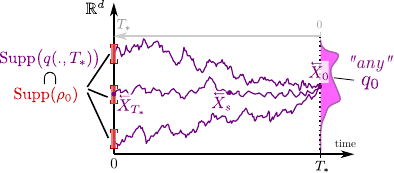}
  \caption{The support of the final distribution $q(.,T_*)$ is contained in the support of the original distribution $ρ_0$ (see Corollary \ref{cor:support}). As a consequence, the {\it exact} reverse diffusion does not generate new samples.}
  \label{fig:support_q_T}
\end{figure}

\begin{corollary}
  \label{cor:support}
  The support of $q(.,T_*)$ is contained in the support of $ρ_0$:
  \begin{equation}
    \label{eq:support}
    \text{Supp}\big(q(.,T_*)\big) ⊂ \text{Supp}(ρ_0).
  \end{equation}
  In other words, if $ρ_0(x)=0$, then $q(x,T_*)=0$.
\end{corollary}
\begin{proof}
  This is a direct consequence of the formula \eqref{eq:final_q}.
\end{proof}
In a practical scenario, the distribution $ρ_0$ is a sum of Dirac masses centered at the original collection of samples $\left\{{\bf x}_i\right\}_{1\leq i\leq N}$:
\begin{equation}
  \label{eq:rho_0_dirac_masses}
  ρ_0(x) = \frac{1}{N} ∑_{i=1}^N δ_{{\bf X}_i}(x).
\end{equation}
From Corollary \ref{cor:support}, we deduce that from any choice of $q_0$, the final distribution $q(.,T_*)$ has to be of the form
\begin{equation}
  \tilde{ρ}_0(x) = ∑_{i=1}^N ω_i\, δ_{{\bf X}_i}(x),
\end{equation}
with $\left\{ω_i\right\}_{1\leq i\leq N}$ a convex combination\footnote{$\omega_i \geq 0$ for all $1\leq i\leq N$ and $\sum_{i=1}^N \omega_i = 1$}. We can compute the weights $ω_i$ explicitly using the formula \eqref{eq:explicit_q_y_T_x_0}, but the key message is that the exact reverse diffusion will not create new samples.

\section{Reverse diffusion in practice}\label{sec:sec5}
\setcounter{equation}{0}

\subsection{Explicit reverse SDE}


We would like to provide a more practical expression of the reverse dynamics. Recall that the process $\cev{X}_s$ has its generator $R$ given by \eqref{eq:R}. Consequently, we conclude that $\cev{X}_s$ is driven by the following It\^o diffusion:

\begin{equation} \label{eq:rev_dynam}
  \dd\cev{X}_s = \cev{X}_s\,\dd s + 2\,\nabla \ln \rho\left (\cev{X}_s, T_*-s\right)\,\dd s + \sqrt{2}\,\dd B_s,
\end{equation}
where $\nabla \ln \rho$ is known as the {\it score} \cite{song2019generative} of the distribution $\rho$. To get a more explicit expression of the score $∇ \ln ρ$, we consider the law of $X_0$ conditioned on $X_t = x$ and denote the expected value
\begin{equation}
  \label{eq:exp_x0}
  \overline{\bf x}_0(x, t) = \mathbb{E}\left[X_0 \mid X_t = x\right].
\end{equation}
In analogy with the method of characteristics encountered in the study of transport-type PDEs,  $\overline{\bf x}_0$ represents the average ``foot'' of the characteristics passing through $x$ at time $t$ (see figure \ref{fig:lighthouse} for an illustration). Using the pairwise distribution $ρ(.\mid .)$ \eqref{eq:joint_law}, we can express $\overline{\bf x}_0$ explicitly and deduce an expression for the score.

\begin{lemma}
  \label{lem:x_hat}
  If $X_0 \sim \rho_0$, we can recast $\overline{\bf x}_0$ as
  \begin{equation}
    \label{eq:exp_x0_explicit}
    \overline{\bf x}_0(x, t) = \frac{\int_{y∈ℝ^d} y\,\rho(x, t \mid y, 0)\,\rho_0(y)\, \dd y}{\rho(x, t)},
  \end{equation}
  and the score $∇ \ln ρ$ as:
\begin{equation}
  ∇ \ln ρ(x,t) = \frac{\expo^{-t}\,\overline{\bf x}_0(x, t)-x}{1-\expo^{-2t}} .
  \label{eq:exp_grad_ln_rho}
\end{equation}
\end{lemma}
\begin{proof}
  By Bayes' rule, we have
  \begin{displaymath}
    \rho(y, 0 \mid x, t) = \frac{\rho(x, t \mid y, 0)\,\rho_0(y)}{\rho(x, t)}.
  \end{displaymath}
  We then apply this equality in the definition of $\overline{\bf x}_0$:
  \begin{equation*}
    \overline{\bf x}_0(x, t) = \mathbb{E}\left[X_0 \mid X_t = x\right] = \int_{y∈ℝ^d}y\,\rho(y, 0 \mid x, t)\,\dd y
    = \frac{\int_{y∈ℝ^d} y\,\rho(x, t \mid y, 0)\,\rho_0(y)\, \dd y}{\rho(x, t)}.
  \end{equation*}
  Since $\rho(x, t \mid y, 0)$ is an explicit Gaussian density, we deduce:
  \begin{equation}
   \nabla_x\rho(x, t \mid y, 0) = \rho(x, t \mid y, 0)\cdot \left (-\frac{x-\alpha_ty}{\beta_t}\right )
    \label{eq:grad_rho}
  \end{equation}
  with $α_t=\expo^{-t}$ and $β_t=1-\expo^{-2t}$. Therefore we can rewrite the score $\nabla \ln \rho$ as follows:
  \begin{align*}
    \nabla \ln \rho(x, t ) &= \frac{\nabla \rho(x, t)}{\rho(x, t)}=\frac{\int_{y∈ℝ^d}\nabla_x \rho(x, t \mid y, 0)\,\rho_0(y)\,\dd y}{\rho(x, t)}\\
                           &= \frac{\int_{y∈ℝ^d}\rho(x, t \mid y, 0) \left (-\frac{x-\alpha_ty}{\beta_t}\right )\,\rho_0(y)\,\dd y}{\rho(x, t)} \\
                           &= \frac{-\frac{x}{\beta_t}\int_{y∈ℝ^d}\rho(x, t \mid y, 0)\, \rho_0(y)\,\dd y +\frac{\alpha_t}{\beta_t}\int_{y∈ℝ^d}y\,\rho(x, t \mid y, 0)\, \rho_0(y)\,\dd y}{\rho(x, t)} \\
                           &=-\frac{x}{\beta_t} + \frac{\alpha_t}{\beta_t}\,\overline{\bf x}_0(x, t).
  \end{align*}
\end{proof}

Thanks to the previous lemma, we can express the reverse diffusion process in terms of $\overline{\bf x}_0$.

\begin{proposition}
  \label{prop:xhat_dynam}
  The reverse diffusion dynamics \eqref{eq:rev_dynam} can be rewritten as
  \begin{equation}
    \label{eq:reverse_SDE_alternative_form}
    \dd\cev{X}_s = \cev{X}_s\,\dd s + 2\,\frac{\alpha_{t}\,\overline{\bf x}_0\left (\cev{X}_s, t\right ) - \cev{X}_s}{\beta_{t}}\,\dd s + \sqrt{2}\,\dd B_s,
  \end{equation}
  with $t=T_*-s$, $\alpha_t = \expo^{-t}$, $\beta_t = 1 - \expo^{-2t}$ and $\overline{\bf x}_0$ given by \eqref{eq:exp_x0}.
\end{proposition}

\begin{figure}[ht]
  \centering
  \includegraphics[width=.8\textwidth]{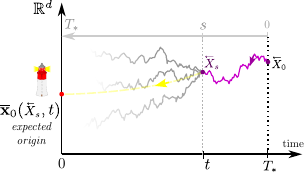}
  \caption{Solving the reverse diffusion written as a SDE \eqref{eq:reverse_SDE_alternative_form} requires computing  $\overline{\bf x}_0$ \eqref{eq:exp_x0}, the expected ``foot'' of the characteristics passing through point $x$ at time $t$. This information, like a lighthouse, guides the process $\cev{X}_s$ back to $ρ_0$.}
  \label{fig:lighthouse}
\end{figure}

\begin{remark}
  When $t \gg 1$ is large, the reverse diffusion reduces to the forward diffusion process \eqref{eq:SDE}:
  \begin{equation}
    \label{eq:reverse_SDE_simplified}
    \dd\cev{X}_s \stackrel{t \gg 1}{\approx} \cev{X}_s\,\dd s + 2\,\frac{{\bf 0} - \cev{X}_s}{1}\,\dd s + \sqrt{2}\,\dd B_s \;=\; -\cev{X}_s\,\dd s + \sqrt{2}\,\dd B_s.
  \end{equation}
  However, when $t\approx 0$, the reverse diffusion becomes a stiff relaxation towards $\overline{\bf x}_0$:
  \begin{equation}
    \label{eq:reverse_SDE_t0}
    \dd\cev{X}_s \stackrel{t \sim 0}{\approx} \frac{1}{t}\left(\overline{\bf x}_0 - \cev{X}_s \right)\,\dd s +  \mathcal{O}(1)\,\dd s + \sqrt{2}\,\dd B_s,
  \end{equation}
  with $\overline{\bf x}_0=\overline{\bf x}_0\left (\cev{X}_s, T_*-s\right )$ and $t=T_*-s$. If we neglect the $\mathcal{O}(1)$ term and assume that $\overline{\bf x}_0$ is constant, then this SDE can be explicitly solved by applying the It\^{o}'s formula to $\cev{X}_s/(T_*-s)$:
  \begin{equation}
    \dd\cev{X}_s = \frac{1}{t}\left(\overline{\bf x}_0 - \cev{X}_s \right)\,\dd s + \sqrt{2}\,\dd B_s \;\; ⇒ \;\; \cev{X}_s = \cev{X}_0 + \frac{t}{T_*}(\overline{\bf x}_0 - \cev{X}_0) + Z_s
  \end{equation}
  where $Z_s= \sqrt{2}\,t\,∫_{r=0}^s \frac{\dd B_r}{T_*-r}$ is a centered Gaussian distribution with covariance matrix $(2\,t\,s/T_*)⋅\text{Id}$. 
  Under these simplications, the process $\cev{X}_s$ converges linearly to $\overline{\bf x}_0$ as $s→T_*$ (i.e., $t→0$). It is similar to a Brownian bridge interpolating $\cev{X}_0$ and $\overline{\bf x}_0$ on the interval $s∈[0,T_*]$.
\end{remark}

\begin{remark}
  We can also rewrite the reverse diffusion process in ways that involve hyperbolic trigonometric functions. For instance, we have \begin{equation}\label{eq:reverse_SDE_alternative1}
    \dd \cev{X}_s = \frac{1}{\tanh(T_*-s)}\,\left(\frac{\overline{\bf x}_0\left(\cev{X}_s,T_*-s\right)}{\cosh(T_*-s)} - \cev{X}_s\right)\,\dd s + \sqrt{2}\,\dd B_s.
  \end{equation}
  Using the identity $\tanh(x/2) = \frac{\cosh(x)-1}{\sinh(x)}$ and the definition $\csch(x) = \frac{1}{\sinh(x)}$, equation \eqref{eq:reverse_SDE_alternative1} can also be rewritten as
  \begin{equation}\label{eq:reverse_SDE_alternative2}
    \begin{aligned}
      \dd\cev{X}_s &=-\tanh\left(\frac{T_*-s}{2}\right)\,\cev{X}_s\,\dd s+\csch(T_*-s)\left(\overline{\bf x}_0\left(\cev{X}_s,T_*-s\right)-\cev{X}_s\right)\,\dd s \\
                   &\qquad + \sqrt{2}\,\dd B_s.
    \end{aligned}
  \end{equation}
\end{remark}

\subsection{Simplified case: \texorpdfstring{$ρ_0$}{rho 0} Dirac (unique sample)}

The explicit reverse dynamics \eqref{eq:reverse_SDE_alternative_form} is in general nonlinear and therefore not explicitly solvable. 
However, if we suppose that all the forward trajectories $X_t$ originate from the same origin $x_0$ (i.e., $ρ_0=δ_{x_0}$), then $\overline{\bf x}_0$ \eqref{eq:exp_x0} is constant and independent of $\cev{X}_s$. We can solve the reverse dynamics explicitly in this case, which will be useful in designing a numerical scheme afterwards.

\begin{theorem}\label{thm:3}
  If $X_0 = x_0$ is fixed (i.e., if $\rho_0 = \delta_{x_0}$ for some $x_0 \in \mathbb{R}^d$), then
  \begin{equation}\label{eq:explicit_SDE_soln}
    \cev{X}_s = \frac{\sinh (T_*-s)}{\sinh(T_*)}\,\cev{X}_0 + \frac{\sinh (s)}{\sinh(T_*)}\,x_0 + \sqrt{2\sinh(T_*-s)\frac{\sinh(s)}{\sinh(T_*)}}\,\epsilon,
  \end{equation}
  where $\epsilon \sim \mathcal{N}({\bf 0},\text{Id})$ is a standard Gaussian random vector in dimension $d$. 
\end{theorem}

\begin{proof}
  The proof is based on the search for a suitable integrating factor. To begin with, we define $u(s) = -\ln\left (\frac{\sinh(T_*-s)}{ \sinh(T_*)}\right)$ and thus $\expo^{u(s)} = \frac{\csch (T_*-s )}{ \csch(T_* )}$.
  Multiplying both sides of the SDE \eqref{eq:reverse_SDE_alternative2} by $\expo^{u(s)}$ and rearranging the resulting equation, we obtain
  \begin{align*}
    &\expo^{u(s)}\,\dd\cev{X}_s + \expo^{u(s)}\,\left(\tanh\left (\frac{T_*-s}{2}\right ) + \csch (T_*-s)\right)\,\cev{X}_s\,\dd s \\
    &\qquad \qquad = \expo^{u(s)}\,\csch (T_*-s)\,\overline{\bf x}_0\left (\cev{X}_s, T_*-s\right )\,\dd s + \expo^{u(s)}\,\sqrt{2}\,\dd B_s.
  \end{align*}
  Since
  \begin{equation*}
    \frac{\dd}{\dd s}\left (\expo^{u(s)}\right ) = \expo^{u(s)}\,\left(\tanh\left (\frac{T_*-s}{2}\right ) + \csch (T_*-s)\right),
  \end{equation*}
  we have
  \begin{equation*}
    \dd\left(\expo^{u(s)}\cev{X}_s\right)=\expo^{u(s)}\,\csch(T_*-s)\,\overline{\bf x}_0\left(\cev{X}_s, T_*-s\right)\,\dd s+\expo^{u(s)}\,\sqrt{2}\,\dd B_s.
  \end{equation*}
  Thus
  \begin{equation}\label{eq:explicit_rev1}
    \expo^{u(s)}\,\cev{X}_s = e^{u(0)}\cev{X}_0 + \int_0^s \expo^{u(r)}\,\csch(T_*-r)\,\overline{\bf x}_0\left(\cev{X}_r, T_*-r\right)\,\dd r
    + \int_0^s \expo^{u(r)}\,\sqrt{2}\,\dd B_r.
  \end{equation}
  As $X_0=x_0$ is fixed by our assumption, $\overline{\bf x}_0\left(\cev{X}_r, T_*-r\right)=x_0$, and
  \begin{equation}\label{eq:explicit_rev2}
    \begin{aligned}
      \int_0^s \expo^{u(r)}\,\csch(T_*-r)\,\overline{\bf x}_0\left(\cev{X}_r, T_*-r\right)\, \dd r &= \int_0^s \frac{\csch^2(T_*-r)}{\csch(T_*)}\,x_0\,\dd r\\
                                                                                                   &= \frac{\coth(T_*-s) - \coth(T_*)}{\csch(T_*)}\,x_0.
    \end{aligned}
  \end{equation}
  Inserting \eqref{eq:explicit_rev2} into \eqref{eq:explicit_rev1} and rearranging, we find
  \begin{equation}\label{eq:explicit_rev3}
    \begin{aligned}
      &\cev{X}_s = \expo^{-u(s)}\,\expo^{u(0)}\,\cev{X}_0 + \expo^{-u(s)}\,\frac{\coth(T_*-s) - \coth(T_*)}{\csch(T_*)}\,x_0 + \expo^{-u(s)}\,\int_0^s \expo^{u(r)} \,\sqrt{2}\,\dd B_r \\
      &= \frac{\sinh (T_*-s )}{ \sinh(T_*)}\,\cev{X}_0 + \left(\frac{\sinh (T_*-s )}{\sinh(T_*)}\right)\left(\frac{\coth(T_*-s) - \coth(T_*)}{\csch(T_*)}\right)\,x_0 \\
      &\qquad +\frac{\sinh (T_*-s )}{ \sinh(T_*)}\,\int_0^s \expo^{u(r)}\,\sqrt{2}\,\dd B_r \\
      &= \frac{\sinh(T_*-s)}{\sinh(T_*)}\,\cev{X}_0 + \sinh(T_*-s)\,\left(\coth(T_*-s) -\coth(T_*)\right)\,x_0\\
      &\qquad +\frac{\sinh(T_*-s)}{ \sinh(T_*)}\,\int_0^s \expo^{u(r)}\,\sqrt{2}\,\dd B_r.
    \end{aligned}
  \end{equation}
  Notice that $\sinh(T_*-s)\,\left(\coth(T_*-s)-\coth(T_*)\right) = \frac{\sinh(s)}{\sinh(T_*)}$, from which we arrive at the following simplified expression for the variance of (each independent component of) the It\^o integral in \eqref{eq:explicit_rev3}:
  \begin{equation}\label{eq:explicit_rev5}
    \begin{aligned}
      \frac{\sinh^2(T_*-s )}{\sinh^2(T_* )}\,\int_0^s 2\,\expo^{2\,u(r)}\,\dd r &= 2\,\frac{\sinh^2 (T_*-s )}{ \sinh^2(T_* )}\,\int_0^s \frac{\csch^2 (T_*-s)}{\csch^2(T)}\,\dd r \\
                                                                                &= 2\,\left (\frac{\sinh^2 (T_*-s )}{ \sinh^2(T_* )}\right )\left (\frac{\coth (T_*-s) - \coth(T_*)}{\csch^2(T_*)}\right) \\
                                                                                &= 2\,\sinh (T_*-s )\,\frac{\sinh(s)}{\sinh(T_*)}.
    \end{aligned}
  \end{equation}
  As $\cev{X}_s$ is a $d$-dimensional Gaussian random vector with independent components (due to the non-randomness of $\overline{\bf x}_0$), we obtain \eqref{eq:explicit_SDE_soln}.
\end{proof}

As an immediate corollary from Theorem \ref{thm:3}, we deduce that the law $q(.,s)$ of $\cev{X}_s$ converges in Wasserstein distance toward the initial distribution $\rho_0=\delta_{x_0}$:
\begin{equation}\label{eq:W2_convergence}
  W_2\left(q(.,s),\rho_0\right) = W_2\left(q(.,s),\delta_{x_0}\right) \xrightarrow{s \to T} 0
\end{equation}
for any $q_0 = q(.,0) \in \mathcal{P}(\mathbb{R}^d)$ with a finite second moment, where $W_2(.,.)$ represents the usual Wasserstein distance of order $2$ between probability measures on $\mathbb{R}^d$. Indeed, it suffices to show \eqref{eq:W2_convergence} in one space dimension (i.e., when $d=1$) thanks to the independence of components of the $d$-dimensional Gaussian $\cev{X}_s$. By our explicit formula \eqref{eq:explicit_SDE_soln} for $\cev{X}_s$, we readily see that $\mathbb E[\cev{X}_s] \xrightarrow{s \to T} x_0$ and $\textrm{Var}[\cev{X}_s] \xrightarrow{s \to T} 0$, whence
\begin{equation}
  \label{eq:wasserstein_rho_0}
  W^2_2\big(q(.,s),\delta_{x_0}\big) \leq \mathbb{E}[|\cev{X}_s-x_0|^2] \,=\, \mathbb E[(\cev{X}_s)^2]-2\,\mathbb E[\cev{X}_s]\,x_0 + x^2_0 \;\xrightarrow{s \to T} 0.
\end{equation}

\subsection{Score-matching \& stable diffusion}

From the SDE \eqref{eq:reverse_SDE_alternative_form}, we would like to derive a numerical scheme to update $\cev{X}_s$ at discrete times. The main challenge is to estimate $\overline{\bf x}_0$ \eqref{eq:exp_x0_explicit} since $ρ_0$ is unknown (we only have access to the sample $\{{\bf x}_i\}_{i=1}^N$). We will discuss this task in the following subsection. Let's focus here on the numerical scheme assuming that we can compute $\overline{\bf x}_0$ in {\it some way}.

Let's denote a partition $0=s_0<s_1 <…<s_{n_*}=T_*$ with $Δs_{n}=s_{n+1}-s_n$. The goal is to find a recursive formula to express $\cev{X}_{s_{n+1}}$ as a function of $\cev{X}_{s_{n}}$.  A simple but somewhat naive method consists in applying the Euler-Maruyama method \cite{kloeden1992stochastic} to the reverse SDE \eqref{eq:reverse_SDE_alternative_form}:
\begin{equation}
  \label{eq:euler}
  \begin{array}{c}
    {\bf Euler}\\
    {\bf Maruyama}
  \end{array} {\bf :} \quad
  \cev{X}_{s_{n+1}}= \cev{X}_{s_n} + Δs_{n}⋅\left( \cev{X}_{s_n} + 2\,\frac{α_{t_n}\,\overline{\bf x}_{0}(\cev{X}_{s_n},t_n) - \cev{X}_{s_n}}{β_{t_n}}\right) + \sqrt{Δs_n}\, \epsilon_n
\end{equation}
with $t_n=T_*-s_n$, $α_{t_n}=\expo^{-t_n}$, $β_{t_n}=1 - \expo^{-2t_n}$  and $\epsilon_n \sim \mathcal{N}(\bf 0,\text{Id})$.

The Euler-Maruyama scheme assumes that the right hand side of the SDE \eqref{eq:reverse_SDE_alternative_form} is frozen over the time interval $[s_n,s_{n+1}]$. A more subtle method is to assume only $\overline{\bf x}_0$ is constant over the time interval $[s_n,s_{n+1}]$. Denoting this constant value by $⟨\overline{\bf x}_0⟩$, we can simplify the nonlinear SDE \eqref{eq:reverse_SDE_alternative_form} into a linear SDE as in \eqref{eq:reverse_SDE_alternative2}:
\begin{equation}
  \label{eq:linear_SDE}
  \dd\cev{X}_s =-\tanh\left(\frac{t}{2}\right)\,\cev{X}_s\,\dd s+\csch(t)\left(⟨\overline{\bf x}_0⟩-\cev{X}_s\right)\,\dd s + \sqrt{2}\,\dd B_s.
\end{equation}
Then we can apply the same strategy as in Theorem \ref{thm:3} to integrate exactly over the time interval $[s_n,s_{n+1}]$ via integrating factor:
\begin{equation}
  \label{eq:euler_better}
  \begin{array}{c}
    {\bf Stable}\\
    {\bf Diffusion}
  \end{array} {\bf :} \quad
  \ds
  \begin{array}{l}
    \ds \cev{X}_{s_{n+1}} = \frac{\sinh (t_{n+1})}{\sinh(t_n)}\,\cev{X}_{s_n} + \frac{\sinh (Δs_n)}{\sinh(t_{n})}\,⟨\overline{\bf x}_0⟩\\
    \\
    \ds \hspace{3cm}+\; \sqrt{2\sinh(t_{n+1})\frac{\sinh(Δs_n)}{\sinh(t_{n})}}\,\epsilon_n,
  \end{array}
\end{equation}
with $t_n=T_*-s_n$ (notice $t_{n+1}<t_n$). We emphasize that the formula \eqref{eq:euler_better} is exact only because we have assumed that $\overline{\bf x}_0$ is constant over the time interval $[s_n,s_{n+1}]$.

The recursive formula \eqref{eq:euler_better} appears quite convoluted, but it corresponds exactly to the formula used in the stable diffusion \cite{ho2020denoising}. Indeed, the derivation of the recursive formula for $\cev{X}_s$ is also exact when we assume $\overline{\bf x}_0$ is given.

\begin{proposition}
  The recursive formula \eqref{eq:euler_better} is equivalent to
  \begin{equation}
    \label{eq:stable_diffusion}
    \cev{X}_{s_{n+1}} = μ_{n+1} + σ_{n+1} \epsilon_n,
  \end{equation}
  with $\epsilon_n \sim \mathcal{N}({\bf 0},\text{Id})$ a standard Gaussian and
  \begin{eqnarray}
    \label{eq:mu_n}
    μ_{n+1} &=& \frac{\expo^{-Δs_n}(1-\expo^{-2t_{n+1}})}{1-\expo^{-2t_n}}\cev{X}_{n} + \frac{\expo^{-t_{n+1}}(1-\expo^{-2Δs_n})}{1-\expo^{-2t_n}}⟨\overline{\bf x}_0⟩, \\
    \label{eq:sig_n}
    σ_{n+1}^2 &=& \frac{(1-\expo^{-2t_{n+1}})(1-\expo^{-2Δs_n})}{1-\expo^{-2t_n}}.
  \end{eqnarray}
\end{proposition}
One can replace $Δs_n$ by $|Δt_n|$ since $|Δt_n|=|t_{n+1}-t_n|=|-s_{n+1}+s_n| = Δs_n$.
\begin{proof}
  Without loss of generality and for the sake of notational simplicity, we work with $d = 1$. Let us start from the formula \eqref{eq:euler_better}:
  \begin{eqnarray*}
    \E[\cev{X}_{s_{n+1}}\mid \cev{X}_{s_{n}}] &=& \frac{\sinh (t_{n+1})}{\sinh(t_n)}\,\cev{X}_{s_n} + \frac{\sinh (Δs_n)}{\sinh(t_{n})}\,⟨\overline{\bf x}_0⟩ \\
                                             &=& \frac{\expo^{t_n-Δs_n}-\expo^{-t_n+Δs_n}}{\expo^{t_n} -\expo^{-t_n}} \,\cev{X}_{s_n} + \frac{\expo^{Δs_n}-\expo^{-Δs_n}}{\expo^{t_n} -\expo^{-t_n}} ⟨\overline{\bf x}_0⟩ \\
    \nonumber &=& \frac{\expo^{-Δs_n}(1-\expo^{-2t_n+2Δs_n})}{1 -\expo^{-2t_n}} \,\cev{X}_{s_n} + \frac{\expo^{-t_n}\expo^{Δs_n}(1-\expo^{-2Δs_n})}{1-\expo^{-2t_n}} ⟨\overline{\bf x}_0⟩,\\
    \text{Var}[\cev{X}_{s_{n+1}}\mid \cev{X}_{s_{n}}] &=& 2\sinh(t_{n+1})\frac{\sinh(Δs_n)}{\sinh(t_{n})}\\
                                             &=& (\expo^{t_{n+1}} - \expo^{-t_{n+1}})\frac{\expo^{Δs_n} - \expo^{-Δs_n}}{\expo^{t_n} - \expo^{-t_n}} \\
                                             &=& \frac{\expo^{t_{n+1}}(1 - \expo^{-2t_{n+1}})\expo^{Δs_n}(1 - \expo^{-2Δs_n})}{\expo^{t_n}(1 - \expo^{-2t_n})}.
  \end{eqnarray*}
\end{proof}
We summarize our discussion in the figures \ref{fig:scheme_sable} and \ref{fig:notation_scheme}. The derivation of the stable diffusion algorithm \eqref{eq:stable_diffusion} using the linear SDE \eqref{eq:linear_SDE} is more complicated than the standard method using the PDE and Bayes' formula. However, one possible advantage is that higher-order numerical schemes might be employed \cite{kloeden1992stochastic} to solve the SDE more accurately.

\begin{figure}[p]
  \centering
  \includegraphics[width=.99\textwidth]{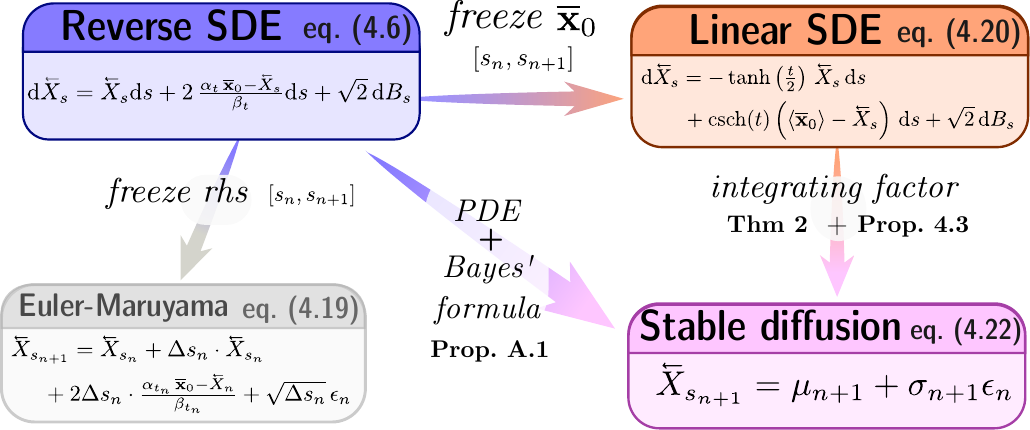}
  \caption{Derivation of the discrete algorithm used in stable diffusion from the reverse SDE \eqref{eq:reverse_SDE_alternative_form}. Assuming $\overline{\bf x}_0$ constant over a time interval $[s_n,s_{n+1}]$, the reverse SDE \eqref{eq:reverse_SDE_alternative_form} is integrable \eqref{eq:explicit_SDE_soln}, leading to the discrete formula \eqref{eq:stable_diffusion}. Another way to derive the numerical scheme is to apply the Bayes' formula to the forward PDE (see Appendix \ref{sec:appendix_stable_diffusion} for details).}
  \label{fig:scheme_sable}
\end{figure}

\begin{figure}[p]
  \centering
  \includegraphics[width=.8\textwidth]{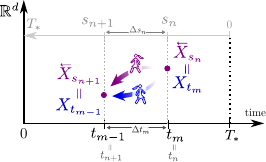}
  \caption{Two ways to derive the numerical scheme of the stable diffusion \eqref{eq:stable_diffusion}. First,  as a forward discretization of the reverse dynamics \eqref{eq:reverse_SDE_alternative_form} represented in magenta in the figure (see also figure \ref{fig:scheme_sable}). Second, as a {\it backward-in-time} discretization of the forward dynamics \eqref{eq:SDE} represented in blue in the figure (see also Appendix \ref{sec:appendix_stable_diffusion}).}
  \label{fig:notation_scheme}
\end{figure}




\subsection{Approximative flow and loss function}

Whether we want to solve the reverse dynamics \eqref{eq:reverse_SDE_alternative_form} (score-matching) or to invert the forward process \eqref{eq:stable_diffusion} (stable diffusion), the main difficulty relies on the estimation of the score $∇ \ln ρ(x,t)$ or equivalently $\overline{\bf x}_0$ \eqref{eq:exp_x0_explicit}. Since $ρ_0$ is unknown, the explicit formula for $\overline{\bf x}_0$ \eqref{eq:exp_x0_explicit} cannot be used. Instead, the usual approach \cite{song2019generative} is to train a neural network (or any type of parametrized function), denoted by ${\bf s}_θ$, to approximate the score $∇ \ln ρ$:
\begin{equation}
  \label{eq:score_matching}
  {\bf s}_θ(x,t) \approx ∇ \ln ρ(x,t) \qquad \text{for } 0\leq t \leq T_*,\,\, x∈ℝ^d.
\end{equation}
The parameters $θ$ are then obtained by minimizing the mean-square error over the trajectories:
\begin{equation}
\ell[θ] = \mathop{\mathbb{E}}_{\substack{t\sim U([0,T_*])\\X_t\sim ρ_t}} \Big[\big|{\bf s}_θ(X_t,t) - ∇ \ln ρ(X_t,t)\big|^2\Big],
  \label{eq:loss}
\end{equation}
The explicit expression for $∇ \ln ρ$ \eqref{eq:exp_grad_ln_rho} motivates the following ``Ansatz'' for ${\bf s}_θ$:
\begin{equation}
  {\bf s}_θ(x,t) = \frac{\expo^{-t}\overline{\bf x}_θ(x, t)-x}{1-\expo^{-2t}}
  \label{eq:ansatz_s}
\end{equation}
where now the neural network $\overline{\bf x}_θ$ aims to approximate $\overline{\bf x}_0$. Then the loss \eqref{eq:loss} becomes 
\begin{equation}
  \ell[θ] = \mathop{\mathbb{E}}_{\substack{t\sim U([0,T_*])\\X_t\sim ρ_t}} \Big[c^2(t)\big(\overline{\bf x}_θ(X_t,t) - \overline{\bf x}_0(X_t,t)\big)^2\Big] \quad \text{ with }c(t)=\frac{\expo^{-t}}{1-\expo^{-2t}}.
  \label{eq:loss2}
\end{equation}
Given that $\overline{\bf x}_0$ is already an expectation \eqref{eq:exp_x0}, we can use the bias-variance decomposition to rewrite the loss as:
\begin{equation}
  \ell[θ] = \mathop{\mathbb{E}}_{\substack{t\sim U([0,T_*])\\X_0\sim ρ_0}} \Big[c^2(t)\big(\overline{\bf x}_θ(X_t,t)- X_0\big)^2\Big] \;-\; \mathop{\text{Var}}_{\substack{t\sim U([0,T_*])\\X_0\sim ρ_0}}[c(t) \overline{\bf x}_0(X_t,t)],
  \label{eq:loss3}
\end{equation}
where $X_t$ is given by the explicit formula \eqref{eq:sol_OU}. Since the variance of $\overline{\bf x}_0$ is independent of $θ$, we can ignore it in the minimization and consider only the lefthand side (the {\it total loss}):
\begin{equation}
  \ell_{total}[θ] = \mathop{\mathbb{E}}_{\substack{t\sim U([0,T_*])\\X_0\sim ρ_0}} \Big[c^2(t)\big(\overline{\bf x}_θ(X_t,t)- X_0\big)^2\Big].
  \label{eq:loss_total}
\end{equation}
Notice that the loss $\ell_{total}$ is lower bounded by the variance term in \eqref{eq:loss3}, hence $\ell_{total}$ cannot converge to zero regardless of the neural network used for $\overline{\bf x}_θ$.

Another modification proposed in \cite{ho2020denoising} is to use the explicit solution of the forward process \eqref{eq:sol_OU} to write $X_0$ as a function of $X_t$:
\begin{equation}
  X_t = \mathrm{e}^{-t}X_0 + \sqrt{1-\mathrm{e}^{-2t}} \,\boldsymbol{ε}_0 \;\;\; ⇒\;\;\; X_0 =  \mathrm{e}^{t}X_t - \mathrm{e}^{t}\sqrt{1-\mathrm{e}^{-2t}} \,\boldsymbol{ε}_0,
  \label{eq:X_t_explicit}
\end{equation}
where $\boldsymbol{ε}_0 \sim \mathcal{N}(0,\text{Id})$. Thus, we can use another Ansatz to predict $X_0$:
\begin{equation}
\overline{\bf x}_θ(X_t,t) = \mathrm{e}^{t}X_t - \mathrm{e}^{t}\sqrt{1-\mathrm{e}^{-2t}} \,\overline{\boldsymbol{ε}}_θ(X_t,t)
  \label{eq:x_0_ansatz}
\end{equation}
where now the neural network $\overline{\boldsymbol{ε}}_θ$ aims to predict the added noise $\boldsymbol{ε}_0$ to the original sample. Plugging in the previous formula into the total loss $\ell_{total}$ \ref{eq:loss_total} gives:
\begin{equation}
  \ell_{total}[θ] = \mathop{\mathbb{E}}_{\substack{t\sim U([0,T_*])\\X_0\sim ρ_0\\\boldsymbol{ε}_0\sim \mathcal{N}(0,\text{Id})}} \Big[\frac{1}{1-\mathrm{e}^{-2t}}\big(\overline{\boldsymbol{ε}}_θ(X_t,t)- \boldsymbol{ε}_0\big)^2\Big],
  \label{eq:loss_total_ep}
\end{equation}
where $X_t$ is given by \eqref{eq:X_t_explicit}. Notice that the weight $1/(1-\mathrm{e}^{-2t})$ is singular as $t→0$. To avoid the singularity, the integral is replaced by a finite (Riemann) sum with a non-uniform discretization in time $0=t_0<...<t_{m_*}=T_*$:
\begin{equation}
  \label{eq:loss_score_matching}
  \widetilde{\ell}_{total}[θ] = \mathop{\mathbb{E}}_{\substack{X_0\sim ρ_0\\ k∈[1,m_*]}} \Big[\big(\overline{\boldsymbol{ε}}_θ(X_{t_k},t_k)- \boldsymbol{ε}_0\big)^2\Big],
\end{equation}
where the discretization $Δt_k=t_{k+1}-t_k$ is smaller as $t$ gets close to zero, hence more weight is given to the loss for $t$ close to zero.

After minimization, the score ${\bf s}_θ(x,t)$ is computed in the following three steps:
\begin{itemize}
  \item[i)] Predict the noise level: $\widetilde{ε}=\overline{\boldsymbol{ε}}_θ(x,t)$,
  \item[ii)] Deduce the expectated origin: $\widetilde{x}_0 =  \mathrm{e}^{t}x - \mathrm{e}^{t}\sqrt{1-\mathrm{e}^{-2t}} \,\widetilde{ε}$,
  \item[iii)] Reconstruct the score: $\displaystyle {\bf s}_θ(x,t) = \frac{\expo^{-t} \widetilde{x}_0 -x}{1-\expo^{-2t}}$.
\end{itemize}
Notice that there is no need to reconstruct the score ${\bf s}_θ$ if one resorts to the stable diffusion formulation \eqref{eq:stable_diffusion}.

\subsection{Overfitting regime and kernel formulation}

Whichever loss function we use (i.e. $\ell$ \eqref{eq:loss}, $\ell_{total}$ \eqref{eq:loss_total}, or $\widetilde{\ell}_{total}$ \eqref{eq:loss_score_matching}), the minimizer is explicitly given by $\overline{\bf x}_0$ \eqref{eq:exp_x0} as the expected value is the unique minimizer of the mean square error. The parametric solution ${\bf s}_θ$ can only approach the non-parametric solution $\overline{\bf x}_0$ and,  in some sense, $\overline{\bf x}_0$ is the optimal limit of ${\bf s}_θ$ as the neural network gets infinitely large.

However, the explicit formulation for $\overline{\bf x}_0$ \eqref{eq:exp_x0_explicit} cannot be used as the density $ρ_0$ is unknown. The same problem occurs when estimating the different loss functions previously introduced. In the absence of $\rho_0$, the sampling $X_0\sim ρ_0$ in the loss \eqref{eq:loss_score_matching} is replaced in practice by a uniform sampling over the known samples $\{{\bf x}_i\}_{i=1}^N$. In other words, the empirical risk is used. However, in this specific case, one can compute $\overline{\bf x}_0$ explicitly, leading to a kernel formulation (note that similar formulations were found in \cite{gu2023memorization} and \cite{yi2023generalization}). 

%

\begin{proposition}
  \label{ppo:kernel_formulation}
  Suppose the initial distribution $ρ_0$ is uniform over the samples $\{{\bf x}_i\}_{i=1}^N$, i.e.,
  \begin{equation}
    \label{eq:rho_0_exp}
    ρ_0(x) = \frac{1}{N}∑_{i=1}^N δ_{{\bf x}_i}(x),
  \end{equation}
  where $δ_x$ is the Dirac mass centered at $x$. Then the expected value $\overline{\bf x}_0$ \eqref{eq:exp_x0} is explicit and given by the following kernel formulation:
  \begin{equation}
    \label{eq:x0_empirical}
    \overline{\bf x}_0(x,t)=\sum_{i=1}^N λ_i(x,t)\,{\bf x}_i,
  \end{equation}
  where $\{λ_i(x,t)\}_{i=1}^N$ is a convex combination given by
  \begin{equation}
    \label{eq:alpha_i}
    λ_i(x,t) = \frac{\rho\left(x, t \mid {\bf x}_i, 0\right)}{\sum_{j=1}^N ρ \left(x, t \mid {\bf x}_j, 0\right)}
  \end{equation}
  with $ρ \left(., t\mid y,0\right)\sim \mathcal{N}\left(\expo^{-t}y,β_t^2\text{Id}\right)$ and $β_t^2=1-\expo^{-2t}$.
\end{proposition}
\begin{proof}
  Plugging in the expression of $ρ_0$ \eqref{eq:rho_0_exp} into \eqref{eq:exp_x0_explicit} and then using the fact that
  \begin{equation}
    ρ(x,t) = \int_{y∈ℝ^d} ρ(x, t \mid y, 0)\,\rho_0(y)\, \dd y
  \end{equation}
  finishes the proof.
\end{proof}

The weight $λ_i(x,t)$ \eqref{eq:alpha_i} encodes the probability that $X_0 = {\bf x}_i$ given that $X_t = x$. An alternative handy expression of the vector $\boldsymbol{λ} = (λ_1,\ldots,λ_N$) is given by
\begin{equation}\label{eq:softmax_representation}
  \boldsymbol{λ}(x,t) = \textrm{Softmax}\left(-\frac{|x-\expo^{-t}\,{\bf x}_1|^2}{2β_t^2},\ldots,-\frac{|x-\expo^{-t}\,{\bf x}_N|^2}{2β_t^2}\right),
\end{equation}
with $\textrm{Softmax}$ representing the usual softmax function. Notice that as $t \gg 1$, the weights $λ_i$ become uniform, i.e.,
\begin{equation}
  \label{eq:weight_t_large}
  \boldsymbol{λ}(x,t) \stackrel{t→+∞}{\longrightarrow} \left(\frac{1}{N},…,\frac{1}{N}\right),
\end{equation}
leading $\overline{\bf x}_0$ \eqref{eq:x0_empirical} to become the mean of the samples $\{{\bf x}_i\}_{i=1}^N$. On the other hand, as $t→0$, the vector $\boldsymbol{λ}$ becomes (almost surely) a one-hot vector:
\begin{equation}
  \label{eq:weight_t_small}
  \boldsymbol{λ}(x,t) \stackrel{t→0}{\longrightarrow} (0,…,0,1,0,…,0)
\end{equation}
where the value $1$ indicates the closest sample ${\bf x}_i$ to $x$. We can further formalize this observation by introducting the Voronoi tessellation $\{\mathcal{C}_i\}_{i=1}^N$ associated with the samples $\{{\bf x}_i\}_{i=1}^N$:
\begin{equation}
  \label{eq:voronoi}
  \mathcal{C}_i = \{x∈ℝ^d \mid |x-{\bf x}_i| = \min (|x-{\bf x}_1|,…,|x-{\bf x}_N|)\}.
\end{equation}
Then as $t→0$, the expected value $\overline{\bf x}_0$ becomes constant on the (interior of the) Voronoi cells $C_i$:
\begin{equation}
  \label{eq:x0_voronoi}
  \overline{\bf x}_0(x,t) \stackrel{t→0}{\longrightarrow} {\bf x}_i \quad \text{for } x∈\mathring{C}_i.
\end{equation}
Moreover, the reverse diffusion becomes a stiff relaxation as $t→0$ \eqref{eq:reverse_SDE_t0}. Thus, for $\cev{X}_s∈\mathring{\mathcal{C}}_i$,
\begin{equation}
  \label{eq:reverse_SDE_t0_bis}
  \dd\cev{X}_s \stackrel{t \sim 0}{\approx} \frac{1}{t}\left({\bf x}_i - \cev{X}_s \right)\,\dd s +  \mathcal{O}(1)\,\dd s + \sqrt{2}\,\dd B_s.
\end{equation}
As $t→0$, the process $\cev{X}_s$ becomes more attracted to ${\bf x}_i$ and is more likely to stay inside the Voronoi cell $\mathring{\mathcal{C}}_i$.

The previous observations are rather informal. The Corollary \ref{cor:support} provides a more clear statement: as the support of $ρ_0$ consists of the discrete sample points, the reverse process will eventually converge to one of the original sample (see figure \ref{fig:tesselation} for an illustration):
\begin{equation}
  \cev{X}_s \stackrel{s→T_*}{\longrightarrow} {\bf x}_i.
  \label{eq:cv_sample}
\end{equation}
Thus, solving the minimization problem {\it exactly} is counter-productive: the generative flow does nothing but sample from the original samples. Hence, there is an inherent risk of overfitting regime when applying the reverse diffusion dynamics since there is no inherent regularization unless penalization terms are added.

\begin{figure}
  \begin{center}
    \includegraphics[width=0.95\textwidth]{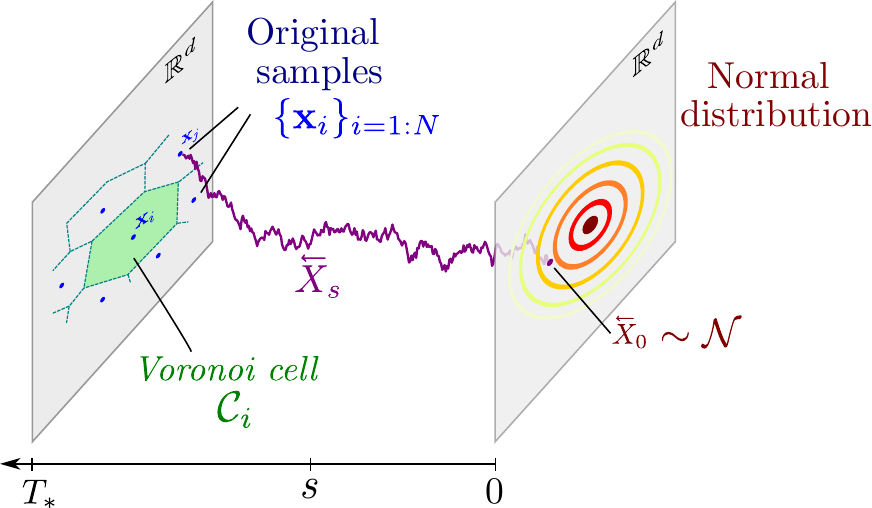}
  \end{center}
  \caption{Starting from a normal distribution (right), the reverse process $\cev{X}_s$ becomes trapped into a Voronoi cell $\mathcal{C}_i$ \eqref{eq:voronoi} as $t→0$ and eventually converges to a sample ${\bf x}_i$ as stated by Corollary \ref{cor:support}.}\label{fig:tesselation}
\end{figure}

\section{Conclusion}\label{sec:conclusion}

In this manuscript, we have shown that reverse diffusion, if executed exactly, does not regularize the original sample distribution $ρ_0$. Indeed, by perfectly reconstructing the forward process, the resulting generated distribution $q_{T_*}$ \eqref{eq:reverse_diffusion} has its support entirely contained within the support of $ρ_0$ (corollary \ref{cor:support}). In other words, generated samples cannot be ``out-of-the-sample." Consequently, for a finite number of samples, exact reverse diffusion will merely reproduce one of the original samples.

In practice, score-matching or stable diffusion \eqref{eq:stable_diffusion}-\eqref{eq:sig_n} does not precisely follow the reverse diffusion process because it introduces a minimization problem \eqref{eq:loss} that is not solved exactly. However, it is relatively straightforward to derive an explicit formula (e.g., using the kernel formulation in Proposition \eqref{eq:rho_0_exp}) for the minimization problem. That said, this computation becomes impractical when the number of original samples $N$ is large. This leads to an apparent paradox: when the minimization problem is solved too accurately, the generative method closely approximates the reverse diffusion process, and thus no new samples are generated (see, e.g.,  \cite{pidstrigach2022score} for a condition under which no generalization occurs). Paradoxically, the minimization problem must be ``imperfect" for the generative method to produce novel samples.

A simple extension of our work would involve demonstrating that (discrete) stable diffusion, in the overfitting regime, indeed converges to the original samples. While we have established this result for the continuous reverse diffusion formulation as a consequence of Corollary \ref{cor:support}, it remains to be rigorously proven for the discrete algorithm.

More broadly, our results raise the question of when generalization occurs when applying reverse diffusion. Reverse diffusion by itself does not ``generalize." A simpler way to express this is to say that {\it ``everything is in the UNet"}. Indeed, for image generation, a UNet architecture is used to estimate the expected origin $\overline{\bf x}_0$ \eqref{eq:exp_x0} (or, equivalently, the expected noise $\boldsymbol{ε}_{0}$). Without it, the overall machinery of the reverse diffusion process would not work. Thus, one has to study the interplay between the reverse dynamics diffusion and the UNet approximation. The flow introduces a mixing property that combines different images intelligently \cite{kadkhodaie2023generalization} without assuming that relevant images lie on a low-dimensional manifold (as in the variational auto-encoder approach). Interestingly, the neural approximation of the score is ``imperfect" enough to generate new data points. Using the kernel formulation in the overfitting regime, one can identify when the practical flow starts to deviate from the exact flow, thereby characterizing the ``good error" necessary for generating new samples.

\clearpage

\newpage
\appendix
\markboth{Appendices}{Appendices}

\section{Formula for stable diffusion}
\label{sec:appendix_stable_diffusion}
\setcounter{equation}{0}

Denote by $ρ$ the law of the forward process \eqref{eq:SDE}. Fix a partition $0=t_0<t_1<...<t_m<...<T_*$ with $Δt_m = t_{m}-t_{m-1}$, and set $X_m = X_{t_m}$. We define $ρ(X_{m}\mid X_{m-1})$ to be the law of $ρ$ at time $t_{m}$ knowing $X_{m-1}$. Using our pairwise distribution \eqref{eq:joint_law}, this is equivalent to
\begin{equation}
  \label{eq:pairwise_stable}
  ρ(X_{m}=y\mid X_{m-1}=x) = ρ(y,t_{m} \mid x,t_{m-1}).
\end{equation}
Our goal is to reverse the equation \eqref{eq:pairwise_stable} and express the probability of the past position $X_{m-1}$ knowing the future $X_{m}$. The trick is once again to assume that we know the starting  position $x_0$.
\begin{proposition}
  \begin{equation}
    \label{eq:stable_formula}
    ρ(X_{m-1} \mid X_m,x_0) \sim \mathcal{N}(\overline{μ}_{m-1},\overline{σ}_{m-1}^2\text{Id})
  \end{equation}
  with
  \begin{eqnarray}
    \label{eq:mu_m_appendix}
    \overline{μ}_{m-1} &=& \frac{\expo^{-Δt_m}(1-\expo^{-2t_{m-1}})}{1-\expo^{-2t_m}}X_m + \frac{\expo^{-t_{m-1}}(1-\expo^{-2Δt_m})}{1-\expo^{-2t_m}}x_0, \\
    \label{eq:sigma2_m_appendix}
    \overline{σ}_{m-1}^2 &=& \frac{(1-\expo^{-2Δt_m})(1-\expo^{-2t_{m-1}})}{1-\expo^{-2t_m}}.
  \end{eqnarray}
\end{proposition}
\begin{proof}
  We start by invoking the Bayes' formula:
  \begin{eqnarray}
    ρ(X_{m-1} \mid X_m,x_0) &=&\frac{ρ(X_m \mid X_{m-1},x_0)⋅ρ(X_{m-1}\mid x_0)}{ρ(X_m\mid x_0)} \\
                            &=& c\,⋅\, ρ(X_m \mid X_{m-1})⋅ρ(X_{m-1}\mid x_0),
  \end{eqnarray}
  using the Markov property and the fact that the numerator is constant in $X_{m-1}$. Notice that the notation is somewhat confusing here since $ρ(X_m \mid X_{m-1})$ has to be seen as a function of $X_{m-1}$ and not $X_m$. Hence,
  \begin{eqnarray*}
    ρ(X_m \mid X_{m-1}) = c \exp\left(-\frac{|X_m-\expo^{-Δt_m} X_{m-1}|^2}{2(1-\expo^{-2Δt_m})}\right) = c \exp\left(-\frac{|\expo^{Δt_m}X_m- X_{m-1}|^2}{2(\expo^{2Δt_m}-1)}\right).
  \end{eqnarray*}
  To conclude, we only need to apply Lemma \ref{lem:product_gauss} with
  \begin{eqnarray}
    ρ(X_m \mid X_{m-1}) &\sim& \mathcal{N}(\expo^{Δt_m}X_m\;,\;(\expo^{2Δt_m}-1)\text{Id}) \\
    ρ(X_{m-1}\mid x_0) &\sim& \mathcal{N}(\expo^{-t_{m-1}}x_0\;,\;(1-\expo^{-2t_{m-1}})\text{Id}),
  \end{eqnarray}
  leading to  $ρ(X_{m-1} \mid X_m,x_0) \sim \mathcal{N}(\overline{μ},\overline{σ}^2\text{Id})$ with
  \begin{eqnarray*}
    \overline{μ} &=& \frac{(1-\expo^{-2t_{m-1}})\expo^{Δt_m}X_m}{(\expo^{2Δt_m}-1) + (1-\expo^{-2t_{m-1}})} \;+\;  \frac{(\expo^{2Δt_m}-1)\expo^{-t_{m-1}}x_0}{(\expo^{2Δt_m}-1) + (1-\expo^{-2t_{m-1}})}.
  \end{eqnarray*}
  We recognize the formula for $\overline{μ}_{m-1}$ \eqref{eq:mu_m_appendix}. Similarly, we have
  \begin{eqnarray*}
    \overline{σ}^2 &=& \frac{(\expo^{2Δt_m}-1)⋅(1-\expo^{-2t_{m-1}})}{(\expo^{2Δt_m}-1) + (1-\expo^{-2t_{m-1}})} = \frac{(\expo^{2Δt_m}-1)⋅(1-\expo^{-2t_{m-1}})}{\expo^{2Δt_m}(1-\expo^{-2t_m})}.
  \end{eqnarray*}
\end{proof}

\begin{lemma}
  \label{lem:product_gauss}
  Suppose $f\sim \mathcal{N}(μ_1,σ_1^2\text{Id})$ and $g\sim \mathcal{N}(μ_2,σ_2^2\text{Id})$. Then
  \begin{equation}
    h(z) = cf(z)g(z) \sim \mathcal{N}(\overline{μ},\overline{σ}^2\text{Id}),
  \end{equation}
  where $c>0$ is a normalizing constant with
  \begin{eqnarray}
    \label{eq:mu_sigma2_lemma}
    \overline{μ} = \frac{σ_2^2 μ_1 + σ_1^2 μ_2}{σ_1^2+ σ_2^2} & \text{ and } &
                                                                               \overline{σ}^2 = \frac{σ_1^2σ_2^2}{σ_1^2+ σ_2^2}.
  \end{eqnarray}
\end{lemma}
\begin{proof}
  We multiply the two Gaussian distributions:
  \begin{eqnarray*}
    f(z)⋅ g(z) &=& c \exp\left(-\frac{|z-μ_1|^2}{2σ_1^2}-\frac{|z-μ_2|^2}{2σ_2^2}\right) \\
               &=& c \exp\left(-\frac{σ_2^2|z-μ_1|^2 + σ_1^2|z-μ_2|^2}{2σ_1^2σ_2^2}\right)  \;=: c \exp\left(-\frac{A}{2σ_1^2σ_2^2}\right).
  \end{eqnarray*}
  Now we complete the square:
  \begin{eqnarray*}
    A &=& (σ_1^2+σ_2^2)|z|^2 - 2σ_2^2⟨μ_1,z⟩ - 2σ_1^2⟨μ_2,z⟩ + σ_1^2μ_1^2 + σ_2^2μ_2^2\\
      &=& s^2|z|^2 - 2s^2 ⟨\overline{μ},z⟩ + σ_1^2μ_1^2 + σ_2^2μ_2^2
  \end{eqnarray*}
  with $s^2=σ_1^2+σ_2^2$ and $\overline{μ}$ given by \eqref{eq:mu_sigma2_lemma}. Thus,
  \begin{eqnarray*}
    A &=& s^2|z-\overline{μ}|^2 \; - s^2|\overline{μ}|^2 +  σ_1^2μ_1^2 + σ_2^2μ_2^2.
  \end{eqnarray*}
  We deduce that
  \begin{eqnarray*}
    f(z)⋅ g(z) &=& \tilde{c}\exp\left(-\frac{s^2|z-\overline{μ}|^2}{2σ_1^2σ_2^2}\right) = \tilde{c}\exp\left(-\frac{|z-\overline{μ}|^2}{2\overline{σ}^2}\right),
  \end{eqnarray*}
  with $\overline{σ}^2$ given by \eqref{eq:mu_sigma2_lemma}.
\end{proof}

\bibliographystyle{plain}
\bibliography{reverse_diff}

\end{document}